\documentclass{article}
\usepackage{palatino}
\usepackage{authblk}
\usepackage{amsmath}
\usepackage{extarrows}
\usepackage{amssymb}
\usepackage{yhmath}
\usepackage[utf8]{inputenc}
\usepackage{tikz-cd}
\usepackage{appendix}
\usetikzlibrary{calc}
\usepackage{geometry}
\usepackage{amsthm}
\usepackage{todonotes}
\usepackage[linkcolor=black]{hyperref}
\usepackage{stmaryrd}
\title{The Geometry of Triangles}
\author{Yining Chen}
\date{May 15, 2022}
\affil{Shandong University \authorcr Email:\ yining\_chen@mail.sdu.edu.cn}

\begin{document}
\maketitle
\large
\begin{abstract}
    In this article we make the concept of a continuous family of triangles precise and prove the moduli functor classifying oriented triangles admits a fine moduli space but the functor classifying non-oriented triangles only admits a coarse moduli space. We hope moduli spaces of triangles can help understand stacks.
\end{abstract}
\begin{center}
    \textbf{One need only understand the stack of triangles to understand stacks. (M. Artin)}
\end{center}
\tableofcontents
\newtheorem{definition}{Definition}[section]
\newtheorem{theorem}[definition]{Theorem}
\newtheorem{plain}[definition]{Proposition}
\newtheorem{lemma}[definition]{Lemma}
\newtheorem{remark}[definition]{Remark}
\newtheorem{example}[definition]{Example}
\newtheorem{exercise}[definition]{Exercise}
\newtheorem{corollary}[definition]{Corollary}
\newtheorem{fact}[definition]{Fact}
\newtheorem{app}[definition]{Application}
\section{The Formalism of Stacks}
In this section we recall some definitions related to stacks and details can be find in \cite{Vis}.

\subsection{Fibered Category}

\ Given a base category \(\mathcal{C}\), a category over \(\mathcal{C}\) actually means a functor \(F:\mathcal{D}\rightarrow \mathcal{C}\). If \(X\in \mathrm{Ob}(\mathcal{C})\), we can define its \textit{fiber} along \(F\) as follows.

\begin{definition}
The \textbf{fiber} of \(X\in \mathrm{Ob}(\mathcal{C})\) along the functor \(F:\mathcal{D}\rightarrow \mathcal{C}\) is a subcategory \(F^{-1}(X)\) consisting of all objects \(Y\in \mathrm{Ob}(\mathcal{D})\) such that \(F(Y)=X\) and morphisms \(f:Y\rightarrow Y'\) in \(\mathcal{D}\) such that \(F(f)=\mathrm{id}_{X}\). We thus say \(f\) is \textbf{lying over} \(\mathrm{id}_{X}\).
\begin{equation}
\notag
\begin{tikzcd}
Y\arrow[r,"f" above]\arrow[d,dash,dotted]&Y'\arrow[d,dash,dotted,"F"]\\
X\arrow[r,"\mathrm{id}_{X}" below] & X
\end{tikzcd}
\end{equation}
\end{definition}

\ The definition above actually defines a subcategory of \(\mathcal{D}\). The slice category \(\mathcal{C}/X\) consists of objects of the form \(X'\rightarrow X\) in \(\mathcal{C}\), and morphisms  \(f:X'\rightarrow X''\) making the following diagram commutative
\begin{equation}
\notag
\begin{tikzcd}
X'\arrow[rr,"f" above] \arrow[dr]&&X''\arrow[ld]\\
&X
\end{tikzcd}
\end{equation}
There is a natural projection \(P_{X}:\mathcal{C}/X\rightarrow \mathcal{C}\) such that \(P_{X}(X'\rightarrow X)=X', P_{X}(f)=f\).

\ A functor between two categories over \(\mathcal{C}\) is a functor \(H:\mathcal{D}\rightarrow \mathcal{E}\), making the following diagram commutative
\begin{equation}
\notag
\begin{tikzcd}
\mathcal{D}\arrow[rr,"H" above] \arrow[dr]&&\mathcal{E}\arrow[ld]\\
&\mathcal{C}
\end{tikzcd}
\end{equation}
We denote the category of categories over \(\mathcal{C}\) by \(\mathbf{Cat}/\mathcal{C}\). 

\begin{theorem}\label{1.2}
There is a fully faithful embedding \(\mathcal{C}\rightarrow \mathbf{Cat}/\mathcal{C}, X\mapsto \mathcal{C}/X\). 
\end{theorem}
\begin{proof}
We should prove this actually defines a functor first. Given any morphism \(f:X\rightarrow Y\) in \(\mathcal{C}\), there is a natural functor \(f_{*}:\mathcal{C}/X\rightarrow \mathcal{C}/Y\) such that
\[f_{*}(Z\stackrel{g}{\rightarrow }X)=f\circ g: Z\rightarrow Y,\ \ \ \ \  f_{*}(Z\stackrel{h}{\rightarrow}Z')=Z\stackrel{h}{\rightarrow}Z'\]
Note that the category \(\mathcal{C}/X\) is over \(\mathcal{C}\), and the functor \(f_{*}\) can make such a diagram commutative which means \(P_{Y}\circ f_{*}=P_{X}\).

\ Now we prove this embedding is fully faithful. Given any functor \(H:\mathcal{C}/X\rightarrow \mathcal{C}/Y\) over \(\mathcal{C}\), \(P_{Y}\circ H=P_{X}\), then \(H(Z\rightarrow X)\) will be some \(Z\rightarrow Y\) and 
\[H(Z\rightarrow Z')=Z\rightarrow Z'\]
Hence \(H\) is actually a natural transformation \(\mathrm{Hom}(-,X)\rightarrow \mathrm{Hom}(-,Y)\), which will correspond uniquely to a morphism \(X\rightarrow Y\) by Yoneda's lemma. Therefore
\[\mathrm{Fun}_{\mathcal{C}}(\mathcal{C}/X,\mathcal{C}/Y)\cong \mathrm{Hom}_{\mathcal{C}}(X,Y)\] 
\end{proof}

\ Via this fully faithful embedding, we can identify \(\mathcal{C}\) with a full subcategory of \(\mathbf{Cat}/\mathcal{C}\).

\ We know in algebraic topology the definition of \textit{groupoids}, which are categories with all morphisms being isomorphisms. For any category over \(\mathcal{C}\), \(F:\mathcal{D}\rightarrow \mathcal{C}\), given a morphism \(f:X\rightarrow Y\) in \(\mathcal{C}\), there may not exist a morphism \(f':X'\rightarrow Y'\) in \(\mathcal{D}\) such that \(F(f')=f\). Even if \(f'\) exists, it may not be unique. This motivates us to define the concept of \textit{groupoid fibration}.

\begin{definition}\label{1.3}
A category over \(\mathcal{C}\), \(F:\mathcal{D}\rightarrow \mathcal{C}\) is a \textbf{groupoid fibration} or called \textbf{fibered in groupoids}, if for any \(f:X\rightarrow Y\) in \(\mathcal{C}\) and any \(Y'\in \mathrm{Ob}(\mathcal{D})\) lying over \(Y\), there is a unique \(f':X'\rightarrow Y'\) in \(\mathcal{D}\) lying over \(f\). 
 \begin{equation}
 \notag
 \begin{tikzcd}
 X'\arrow[r," \exists!\ f'" above] \arrow[d,dash,dotted] & Y'\arrow[d,dash,dotted,"F"right]\\
 X\arrow[r,"f" below]& Y
 \end{tikzcd}
 \end{equation}
\end{definition}

\ The meaning of uniqueness is as follows: if \(f'':X''\rightarrow Y'\) in \(\mathcal{D}\) is also lying over \(f\), then there is a unique isomorphism \(\alpha:X''\rightarrow X'\) lying over \(\mathrm{id}_{X}\), such that \(f'\circ \alpha=f''\).
 \begin{equation}
 \notag
 \begin{tikzcd}
 X'' \arrow[d, dash, dotted] \arrow[r,"\exists!\  \alpha" above,"\cong" below] \arrow[rr, bend left, "f''" above] &   X' \arrow[r,"f'" above] \arrow[d, dash, dotted] & Y' \arrow[d,dash,dotted,"F" right]\\
X\arrow[r,"\mathrm{id}" above] \arrow[rr, bend right, "f"below] & X\arrow[r,"f" above]& Y
\end{tikzcd}
\end{equation}

\ Why this kind of functors is called a \textit{groupoid fibration}? It's mainly due to the following theorem

\begin{theorem}\label{1.4}
Assume \(F:\mathcal{D}\rightarrow \mathcal{C}\) is a groupoid fibration, and then for any isomorphism \(f:X\rightarrow Y\) in \(\mathcal{C}\), \(f':X'\rightarrow Y'\) lying over \(f\) will also be an isomorphism in \(\mathcal{D}\). Hence the fiber \(F^{-1}(X)\) is a groupoid.
\end{theorem}
\begin{proof}
Assume \(f'\) is lying over an isomorphism \(f\) in \(\mathcal{C}\), and \(f\) has the inverse \(g:Y\rightarrow X\), such that \(f\circ g=\mathrm{id}_{Y},g\circ f=\mathrm{id}_{X}\). For \(X'\in \mathrm{Ob}(\mathcal{D})\) is lying over \(X\), there is a unique \(g':Y''\rightarrow X'\) lying over \(g\).
\begin{equation}
\notag
\begin{tikzcd}
 Y'' \arrow[d, dash, dotted] \arrow[r,"g'" above] \arrow[rr, bend left, "f'\circ g'" above] &   X' \arrow[r,"f'" above] \arrow[d, dash, dotted] & Y' \arrow[d,dash,dotted]\\
Y\arrow[r,"g" above] \arrow[rr, bend right, "\mathrm{id}_{Y}"below] & X\arrow[r,"f" above]& Y
\end{tikzcd}
\end{equation}

\ Hence \(f'\circ g'\) is lying over \(\mathrm{id}_{Y}\). But \(\mathrm{id}_{Y^{'}}:Y'\rightarrow Y'\) is also lying over \(\mathrm{id}_{Y}\), there is a unique isomorphism \(\alpha:Y''\rightarrow Y'\) lying over \(\mathrm{id}_{Y}\) such that \(\mathrm{id}_{Y^{'}}\circ \alpha =f'\circ g'\), which means \(f'\circ g'=\alpha\) is an isomorphism. Hence \(g'\) has a left inverse and \(f'\) has a right inverse. Dually consider
 \begin{equation}
 \notag
 \begin{tikzcd}
 X'' \arrow[d, dash, dotted] \arrow[r,"f''" above] \arrow[rr, bend left, "g'\circ f''" above] &   Y'' \arrow[r,"g'" above] \arrow[d, dash, dotted] & X'\arrow[d,dash,dotted]\\
X\arrow[r,"f" above] \arrow[rr, bend right, "\mathrm{id}_{X}"below] & Y\arrow[r,"g" above]& X
 \end{tikzcd}
 \end{equation}

\ The same argument will imply \(f''\) has a left inverse and \(g'\) has a right inverse. Then \(g'\) has a left inverse and a right inverse, hence an isomorphism. But \(f'\circ g'\) is an isomorphism and this means \(f'\) will be an isomorphism.  
\end{proof}

\ The groupoid fibration requires the lifting of a morphism in the base category is unique. Sometimes this condition is hard to check. In the following we introduce the concept of \textit{Cartesian arrows}.

\begin{definition}
\(F:\mathcal{D}\rightarrow \mathcal{C}\) is a category over \(\mathcal{C}\). A morphism \(f':X'\rightarrow Y'\) in \(\mathcal{D}\) lying over \(f:X\rightarrow Y\) in \(\mathcal{C}\) is \textbf{Cartesian}, if for any \(g':Z'\rightarrow Y'\) lying over \(g:Z\rightarrow Y\) and for any \(h:Z\rightarrow X\) with \(f\circ h=g\), there will exist a unique arrow \(h':Z'\rightarrow X'\) lying over \(h\) such that \(f'\circ h'=g'\).
\begin{equation}
\notag
\begin{tikzcd}
 Z' \arrow[d, dash, dotted] \arrow[r,dotted,"\exists!\  h' " above] \arrow[rr, bend left, "g'" above] &   X' \arrow[r,"f'" above] \arrow[d, dash, dotted] & Y' \arrow[d,dash,dotted]\\
Z\arrow[r,"h" above] \arrow[rr, bend right, "g"below] & X\arrow[r,"f" above]& Y
\end{tikzcd}
\end{equation}
\end{definition}

\begin{definition}
A \textbf{fibered category} over \(\mathcal{C}\), is a functor \(F:\mathcal{D}\rightarrow \mathcal{C}\) such that for any morphism \(f:X\rightarrow Y\) in \(\mathcal{C}\) and any object \(Y'\in \mathrm{Ob}(\mathcal{D})\) lying over \(Y\), there is a Cartesian morphism \(f':X'\rightarrow Y'\) in \(\mathcal{D}\) lying over \(f\).
\end{definition}

\ Naturally there is a question. What's the connection between fibered categories and groupoid fibrations? In fact, every fibered category admits a subcategory which is a groupoid fibration.

\begin{lemma}
Given a category over \(\mathcal{C}\) say \(F:\mathcal{D}\rightarrow \mathcal{C}\), isomorphisms in \(\mathcal{D}\) are Cartesian and the composition of Cartesian morphisms is Cartesian.
\end{lemma}
\begin{proof}
The first part is obvious and we only prove the second part.
 
\ We assume \(f':X'\rightarrow Y',g':Y'\rightarrow Z'\) are Cartesian arrows lying over \(f:X\rightarrow Y, g:Y\rightarrow Z\) respectively. Given any morphism \(u':U'\rightarrow Z'\) lying over \(u:U\rightarrow Z\), and a morphism \(h:U\rightarrow X\) such that \(u=g\circ f\circ h\), we should find a unique morphism \(h':U'\rightarrow X'\) lying over \(h\) such that \(u'=g'\circ f'\circ h'\).
\[\begin{tikzcd}
U' \arrow[r, "\exists!\ h'", dotted] \arrow[d, no head, dotted] \arrow[rrr, "u'", bend left] & X' \arrow[r, "f'"] \arrow[d, no head, dotted] & Y' \arrow[r, "g'"] \arrow[d, no head, dotted] & Z' \arrow[d, no head, dotted] \\
U \arrow[r, "h"] \arrow[rrr, "u"', bend right]                                               & X \arrow[r, "f"]                              & Y \arrow[r, "g"]                              & Z                            
\end{tikzcd}\]

\ Since \(g'\) is Cartesian, there is a unique morphism \(v:U'\rightarrow Y'\) lying over \(f\circ h\) such that \(u'=g'\circ v\). Since \(f'\) is Cartesian, there is a unique morphism \(h':U'\rightarrow X'\) lying over \(h\) such that \(v=f'\circ h'\).
\end{proof}

\ Therefore every fibered category \(F:\mathcal{D}\rightarrow \mathcal{C}\) admits a subcategory \(\mathcal{D}_{\mathrm{Car}}\) consisting of objects the same as \(\mathcal{D}\) and morphisms being Cartesian. This subcategory is a groupoid fibration.

\begin{theorem}\label{1.8}
The subcategory \(\mathcal{D}_{\mathrm{Car}}\) defined above is a groupoid fibration. 
\end{theorem}
\begin{proof}
Given a morphism \(f:X\rightarrow Y\) in \(\mathcal{C}\) and an object \(Y'\in \mathrm{Ob}(\mathcal{D})\) lying over \(Y\), there is a Cartesian morphism \(f':X'\rightarrow Y'\) in \(\mathcal{D}_{\mathrm{Car}}\) lying over \(f\). We need to prove \(f'\) is unique in the sense of Definition \ref{1.3}.
 
\ If \(f'':X''\rightarrow Y'\) is also a Cartesian morphism lying over \(f\)
\begin{equation}
\notag
 \begin{tikzcd}
 X'' \arrow[d, dash, dotted] \arrow[r,dotted,"\exists!\  h " above] \arrow[rr, bend left, "f''" above] &   X' \arrow[r,"f'" above] \arrow[d, dash, dotted] & Y' \arrow[d,dash,dotted]\\
X\arrow[r,"\mathrm{id}_{X}" above] \arrow[rr, bend right, "f"below] & X\arrow[r,"f" above]& Y
 \end{tikzcd}
\end{equation}

\ There is a unique arrow \(h:X''\rightarrow X'\) lying over \(\mathrm{id}_{X}\) such that \(f'\circ h=f''\). Now we only need to prove \(h\) is an isomorphism. Dually there is a unique arrow \(h':X'\rightarrow X''\) lying over \(\mathrm{id}_{X}\) such that \(f''\circ h'=f'\). Then \(f'\circ h\circ h'=f'\). Since \(f'\) is Cartesian and \(h\circ h'\) is lying over \(\mathrm{id}_{X}\), which means \(h\circ h'\) is unique, then \(h\circ h'=\mathrm{id}_{X'}\). Dually \(h'\circ h=\mathrm{id}_{X''}\). \(h\) is an isomorphism especially a Cartesian map in \(\mathcal{D}_{\mathrm{Car}}\).
\end{proof}

\ Conversely that every groupoid fibration is a fibered category is easy to prove. According to the discussion above, to some degree a fibered category is equivalent to a groupoid fibration, because in a fibered category we always concern with Cartesian arrows. And there is another theorem connecting them.

\begin{corollary}\label{1.9}
A fibered category \(F:\mathcal{D}\rightarrow \mathcal{C}\) is a groupoid fibration iff all of its fibers are groupoids.
\end{corollary}
\begin{proof}
The part of ``\(\Rightarrow\)" is the Theorem \ref{1.4}. The part of ``\(\Leftarrow\)" comes from the proof of Theorem \ref{1.8} above where \(h\) is automatically an isomorphism by definition.
\end{proof}

\ Why we prefer the concept of fibered categories? Actually we can view a fibered category as a functor in the sense of 2-categories, where roughly speaking the fibered category \(\mathcal{F}:\mathcal{C}\rightarrow \mathbf{Cat}\) sends objects in \(\mathcal{C}\) to categories. And in general, a category is easier to find than a goupoid.

\begin{definition}\label{1.10}
A \textbf{pseudo-functor} \(\Phi\) on a category \(\mathcal{C}\) is defined as follows
\item [(1)] For each object \(S\in \mathrm{Ob}(\mathcal{C})\), \(\Phi S\) is a category.
\item[(2)] For each arrow \(f:T\rightarrow S\) in \(\mathcal{C}\), \(\Phi_f=f^{*}:\Phi S\rightarrow \Phi T\) is a funcotr.
\item[(3)] For each object \(S\in \mathrm{Ob}(\mathcal{C})\) there is a natural isomorphism \(\epsilon_S:\mathrm{id}_{S}^{*}\xrightarrow{\cong}\mathrm{id}_{\Phi S}\).
\item[(4)] For each pair of arrows \(Q\xrightarrow{f}T\xrightarrow{g}S\), there is a natural isomorphism \(\alpha_{f,g}:f^*\circ g^*\xrightarrow{\cong} (g\circ f)^*\). \(\alpha_{f,g}\) is canonical in the sense that
\begin{align*}
    \alpha_{\mathrm{id},g}=\epsilon\circ g^*:\mathrm{id}^*\circ g^*\xrightarrow{\cong}g^{*}\\
    \alpha_{f,\mathrm{id}}=f^*\circ \epsilon:f^*\circ \mathrm{id}^*\xrightarrow{\cong} f^*
\end{align*}
and the following diagram is commutative
\[\begin{tikzcd}
f^*g^*h^* \arrow[d, "{f^*\circ\alpha_{g,h}}"'] \arrow[r, "{\alpha_{f,g}\circ h^*}"] & (gf)^*h^* \arrow[d, "{\alpha_{gf,h}}"] \\
f^*(hg)^* \arrow[r, "{\alpha_{f,hg}}"']                                             & (hgf)^*                               
\end{tikzcd}\]
\end{definition}

\begin{theorem}\label{1.11}
On a base category \(\mathcal{C}\), pseudo-functors and fibered categories are equivalent up to isomorphism.
\end{theorem}

\begin{proof}
At first we associate every fibered category \(F:\mathcal{D}\rightarrow \mathcal{C}\) with a pseudo-functor. For any object \(S\in \mathrm{Ob}(\mathcal{C})\), \(\Phi S\) is defined to be the fiber \(F^{-1}(S)\). For any morphism \(f:T\rightarrow S\) in \(\mathcal{C}\) and object \(S'\in \mathrm{Ob}(\mathcal{C})\) over \(S\) it has a Cartesian lifting \(f'_{S'}:T'\rightarrow S'\) in \(\mathcal{D}\). This means for any object \(S'\) over \(S\), we choose some Cartesian morphism \(f'_{S'}\). Therefore we define a functor \(f^*:\Phi S\rightarrow \Phi T\) sending the object \(S'\) over \(S\) to \(T'\) via the chosen \(f'_{S'}\). 
\[\begin{tikzcd}
T' \arrow[d, no head, dotted] \arrow[r, "\exists!\ f^*u"', dotted] \arrow[rr, "f'_{S'}", bend left] & T'' \arrow[d, no head, dotted] \arrow[rr, "f'_{S''}", bend left] & S' \arrow[r, "u"'] \arrow[d, no head, dotted] & S'' \arrow[d, no head, dotted] \\
T \arrow[r,equal]                                                                                         & T \arrow[r, "f"']                                                & S \arrow[r,equal]                                   & S                             
\end{tikzcd}\]

\ Given any morphism \(u:S'\rightarrow S''\) over \(\mathrm{id}_{S}\), since \(f'_{S''}\) is a Cartesian map and \(u\circ f'_{S'}\) is lying over \(f\), there will exist a unique map \(f^*u:T'\rightarrow T''\) over \(\mathrm{id}_T\). The uniqueness will imply \(f^*:\Phi S\rightarrow \Phi T\) is actually a functor. If \(T=S\) and \(f=\mathrm{id}_S\), the diagram above with \(u\circ f'_{S'}=f'_{S''}\circ f^*u\) will imply the class of \(\{f'_{S'}:T'\rightarrow S'\}\) actually defines a natural isomorphism \(\epsilon:\mathrm{id}^*_{S}\xrightarrow{\cong}\mathrm{id}_{\Phi S}\). Note that the fact \(f'_{S'}\) is an isomorphism comes from Theorem \ref{1.4} and \ref{1.8}.

\ Given a pair of arrows \(Q\xrightarrow{f}T\xrightarrow{g}S\), the functor \(f^*\circ g^*\) is along the Cartesian morphism \(g'_{S'}\circ f'_{T'}\) for any object \(S'\) lying over \(S\) and \((g\circ f)^*\) is along \((g\circ f)'_{S'}\). The two choices of Cartesian maps for every object \(S'\) over \(S\) will naturally induce a unique isomorphism \(\alpha_{f,g}:f^{*}\circ g^*\xrightarrow{\cong}(g\circ f)^*\) for all \(S'\). Such isomorphism \(\alpha_{f,g}\) is a natural transformation which follows from the uniqueness. We leave readers as an exercise to check \(\alpha_{f,g}\) and \(\epsilon\) defined above satisfy the two canonical properties in the axiom (4) of Definition \ref{1.10}. They follows from the uniqueness as well.

\ Note that for a fixed fibered category \(F:\mathcal{D}\rightarrow \mathcal{C}\), different choices of Cartesian morphisms \(f'_{S'}\) will give isomorphic pseudo-functors.

\ Next let us try to describe the converse construction. Given a pseudo-functor \(\Phi\) over \(\mathcal{C}\) we want to associate it with a fiered category \(F:\mathcal{D}\rightarrow \mathcal{C}\).
\[\mathrm{Ob}(\mathcal{D}):=\{(s,S)|S\in \mathrm{Ob}(\mathcal{C}),\ s\in \mathrm{Ob}(\Phi S)\}\]
A morphism \((t,T)\xrightarrow{(u,f)}(s,S)\) is defined to be 
\[\begin{cases}
f:T\rightarrow S\ \text{in \(\mathcal{C}\)}\\
u:t\rightarrow f^{*}s\ \text{in \(\Phi T\)}
\end{cases}\]

\ Now it remains to define the composition of arrows in \(\mathcal{D}\). Given a composable pair of arrows
\[(q,Q)\xrightarrow{(u,f)}(t,T)\xrightarrow{(v,g)}(s,S)\]
a new morphism \((q,Q)\rightarrow (s,S)\) has two parts \(g\circ f:Q\rightarrow S\) and 
\[q\xrightarrow{u}f^*t\xrightarrow{f^*v}(f^*\circ g^*)s\xrightarrow[\cong]{\alpha_{f,g}}(g\circ f)^*s\]

\ We can check this actually defines a category. The identity map for \((s,S)\rightarrow (s,S)\) is actually \((\epsilon^{-1}_s,\mathrm{id}_S)\). For instance the composable map becomes 
\[q\xrightarrow{u}f^*s\xrightarrow{f^*(\epsilon^{-1}_s)}(f^*\circ \mathrm{id}^{*})s\xrightarrow{\alpha_{f,\mathrm{id}}=f^*\circ \epsilon}f^*s\]
which is just \(u\). To check the associativity for 
\[(q,Q)\xrightarrow{(u,f)}(t,T)\xrightarrow{(v,g)}(s,S)\xrightarrow{(w,h)}(q,Q)\]
firstly we describe \((u,f)\circ \big((v,g)\circ(w,h)\big)\).
\[t\xrightarrow{v}g^*s\xrightarrow{g^*w}g^*h^*q\xrightarrow{\alpha_{g,h}}(hg)^*q\]
Then \((u,f)\circ \big((v,g)\circ(w,h)\big)\) is just 
\[q\xrightarrow{u}f^*t\xrightarrow{f^*(\alpha_{g,h}\circ g^*w\circ v)}f^*(hg)^*q\xrightarrow{\alpha_{f,hg}}(hgf)^*q\]
For \(\big((u,f)\circ (v,g)\big)\circ(w,h)\), it's 
\[q\xrightarrow{\alpha_{f,g}\circ f^*v\circ u}(gf)^*s\xrightarrow{(gf)^*w}(gf)^*h^*q\xrightarrow{\alpha_{gf,h}}(hgf)^*q\]
Finally
\begin{align*}
    &\alpha_{f,hg}\circ f^*(\alpha_{g,h}\circ g^*w\circ v)\circ u\\
    =&\alpha_{f,gh}\circ f^*(\alpha_{g,h})\circ f^*g^*w\circ f^*v\circ u\\
    =&\alpha_{gf,h}\circ \alpha_{f,g} h^*\circ f^*g^*w\circ f^*v\circ u,\ \ \ \text{axiom (4) in Definition \ref{1.10}}\\
    =&\alpha_{gf,h}\circ (gf)^*w\circ \alpha_{f,g}\circ f^*v\circ u, \ \ \ \text{\(\alpha_{f,g}\) is a natural transformation}
\end{align*}
The last equation follows from the following commutative diagram
\[\begin{tikzcd}
f^*g^*s \arrow[d, "f^*g^*w"'] \arrow[r, "{\alpha_{f,g}}"] & (gf)^*s \arrow[d, "(gf)^*w"] \\
f^*g^*h^*q \arrow[r, "{\alpha_{f,g}h^*}"']                & (gf)^*h^*q                  
\end{tikzcd}\]

\ The functor \(F:\mathcal{D}\rightarrow \mathcal{C}\) is defined in the obvious manner, \((s,S)\mapsto S\) and \((u,f)\mapsto f\). Now we only need to prove \(F\) is a fibered category. Given any map \(f:T\rightarrow S\) in \(\mathcal{C}\) and any object \((s,S)\) over \(S\), the Cartesian lifting is defined to be 
\[(\mathrm{id},f):(f^*s,T)\longrightarrow (s,S)\]
In fact 
\[\begin{tikzcd}
{(q,Q)} \arrow[d, no head, dotted] \arrow[rr, "{(w,g)}", bend left] \arrow[r, dotted,"{(\theta,h)}"] & {(f^*s,T)} \arrow[r, "{(\mathrm{id},f)}"] \arrow[d, no head, dotted] & {(s,S)} \arrow[d, no head, dotted] \\
Q \arrow[r, "h"] \arrow[rr, "g"', bend right]                                         & T \arrow[r, "f"]                                                     & S                                 
\end{tikzcd}\]
to have \((\mathrm{id},f)\circ (\theta,h)=(w,g)\),
\[q\xrightarrow{\theta}h^*f^*s\xrightarrow{h^*(\mathrm{id})}h^*f^*s\xrightarrow{\alpha_{h,f}}(fh)^*s=g^*s\]
\(w=\alpha_{h,f}\circ \theta\Rightarrow \theta=\alpha_{h,f}^{-1}w\) is unique.
\end{proof}

\ In most cases it's easier to obtain a pseudo-functor than a fibered category directly.

\subsection{Descent Theory}
\begin{definition}
A \textbf{site} is a category \(\mathcal{C}\) (with pullbacks in general) with Grothendieck topology. A Grothendieck topology is a function \(\mathcal{K}:\mathrm{Ob}(\mathcal{C})\rightarrow \mathrm{Ob}(\mathbf{Sets})\) associating every object \(X\) in \(\mathcal{C}\) a set of arrows \(\{\iota_i:U_i\rightarrow X\}\) satisfying the following axioms
\item[(T1)] If \(f:U\xrightarrow{\sim}X\) is an isomorphism, then \(\{f:U\xrightarrow{\sim} X\}\in \mathcal{K}(X)\).
\item[(T2)] If \(\{\iota_i:U_i\rightarrow X\}\) is a covering, then for any morphism \(f:Y\rightarrow X\), \(\{Y\times_XU_i\rightarrow X\}\) is a covering as well.
\item[(T3)] If \(\{\iota_i:U_i\rightarrow X\}\) is a covering for \(X\) and for any \(i\in I\), \(\{\phi_{ij}:U_{ij}\rightarrow U_i\}\) is a covering for \(U_i\) then \(\{\iota_i\circ\phi_{ij}:U_{ij}\rightarrow X\}\) is a covering for \(X\).
\end{definition}

\begin{example}\label{1.13}
\normalfont
In \(\mathbf{Top}\) for any topological space \(X\), \(\{\iota_i:U_i\rightarrow X\}\) is a covering if \(\cup_i\iota_i(U_i)=X\) and \(\iota_i\)'s are open immersions which means \(U_i\rightarrow \iota_{i}(U_i)\) is a homeomorphism with \(\iota_i(U_i)\subseteq X\) is open. Then the Grothendieck topology is just the usual topology for \(X\).
\end{example}

\begin{example}\label{1.14}
\normalfont
We define the \textit{etale topology} for \(\mathbf{Top}\). In \(\mathbf{Top}\), an etale map \(p:E\rightarrow X\) is local heomorphism which means for every point \(e\in E\) there is an open neighborhood \(V\subseteq E\) containing \(e\) such that \(p|V:V\rightarrow p(V)\) is a homeomorphism with \(p(V)\subseteq X\) open.

\ A family of etale maps \(\{p_i:E_i\rightarrow X\}\in \mathcal{K}(X)\) if \(\cup_i p_i(E_i)=X\). (T1) is trivially satisfied since every homeomorphism is an etale map.

\ (T2). We only need to prove etales maps are stable under pullbacks.
\[\begin{tikzcd} 
Y\times_X E \arrow[r] \arrow[d] \arrow[dr, phantom, "\ulcorner", very near start] & E \arrow[d,"p"] \\ Y \arrow[r,"f" below] & X
\end{tikzcd}\]
Given \((y,e)\in Y\times_X E\), \(f(y)=p(e)\) then there will exist an open neighborhood \(U\subseteq E\) containing \(e\) such that \(p|U:U\xrightarrow{\approx}p(U)\). Let \(V=f^{-1}(p(U))\) and we obtain \(q:(V\times U)\cap (Y\times_X E)\rightarrow V\). We prove \(q\) is a heomorphism. At first \(q\) is surjective since for any \(y\in V\), \(f(y)\in p(U)\) and there is a unique \(e\in U\) such that \(p(e)=f(y)\). If \(q(y_1,e_1)=q(y_2,e_2)\) then \(y_1=y_2\Rightarrow f(y_1)=f(y_2)=p(e_1)=p(e_2)\). Since \(p|U\) is a heomorphism \(e_1=e_2\). Finally \(q\) is open since for any \(V'\times U'\subseteq V\times U\), \(q((V'\times U')\cap (Y\times_X E))=V'\) is open.

\ (T3) is also easy to prove because the composition of two etale maps is etale as well.
\end{example}

\begin{definition}[Descent Datum]
Let \(F:\mathcal{D}\rightarrow \mathcal{C}\) be a fibered category over a site \(\mathcal{C}\). A \textbf{descent datum} for \(\mathcal{D}\) over a covering \(\{\iota_i:X_i\rightarrow X\}\) consists of objects \(E_i\) over \(X_i\) and isomorphisms \(\alpha_{ji}:E_i|X_{ij}\rightarrow E_j|X_{ij}\) over \(X_{ij}\) i.e. in the fiber \(F^{-1}(X_{ij})\) where \(X_{ij}=X_i\times_X X_j\) and \(E_i|X_{ij}\) is the pullback of \(E_i\) along \(X_{ij}\rightarrow X_i\).
\[\begin{tikzcd}[row sep=scriptsize, column sep=scriptsize] 
& X_{ijk}\arrow[dr,phantom, "\ulcorner", very near start] \arrow[dl] \arrow[rr] \arrow[dd] & & X_{jk}\arrow[dddl,phantom, "\ulcorner", very near start] \arrow[dl] \arrow[dd] \\
X_{ij}\arrow[rrdd,phantom, "\ulcorner", very near start] \arrow[rr, crossing over] \arrow[dd] & & X_j \\ 
& X_{ik}\arrow[dr,phantom, "\ulcorner", very near start]\arrow[dl] \arrow[rr] & & X_k \arrow[dl] \\ X_i \arrow[rr] & & X \arrow[from=uu, crossing over]\\
\end{tikzcd}\]
\((E_i,\alpha_{ji})\) satisfies the cocycle condition 
\[\begin{tikzcd}
                                       & E_i|X_{ijk} \arrow[ld, "\alpha_{ki}"'] \arrow[rd, "\alpha_{ji}"] &               \\
E_k|X_{ijk} \arrow[rr, "\alpha_{jk}"'] &                                                                  & E_{j}|X_{ijk}
\end{tikzcd}\]
which means \(\alpha_{jk}\circ \alpha_{ki}=\alpha_{ji}\) over \(X_{ijk}\) i.e. in the fiber \(F^{-1}(X_{ijk})\).

\ A descent datum \((E_i,\alpha_{ji})\) is said to be \textbf{effective} if there exists some object \(E\) over \(X\) with isomorphisms \(\alpha_i:E|X_{i}\xrightarrow{\sim} E_i\) over \(X_i\) such that 
\[\alpha_{ji}=(\alpha_j|X_{ij})\circ (\alpha_i|X_{ij})^{-1}\]

\ A morphism \(f:(E_i,\alpha_{ji})\rightarrow (E'_i,\alpha'_{ji})\) consists of a class of morphisms \(f_i:E_{i}\rightarrow E'_i\) over \(X_i\) such that the following diagram is commutative
\[\begin{tikzcd}
E_i|X_{ij} \arrow[d, "\alpha_{ji}"'] \arrow[r, "f_i|X_{ij}"] & E'_{i}|X_{ij} \arrow[d, "\alpha'_{ji}"] \\
E_j|X_{ij} \arrow[r, "f_j|X_{ij}"']                          & E'_j|X_{ij}                            
\end{tikzcd}\]
\end{definition}

\ The definition above defines a category of descent datum for every covering \(\{\iota_i:X_i\rightarrow X\}\). We denote this category by \(F_{\mathrm{des}}(\{\iota_i:X_i\rightarrow X\})\). Then there is a natural functor 
\[F^{-1}(X)\rightarrow F_{\mathrm{des}}(\{\iota_i:X_i\rightarrow X\}),\ E\mapsto (E|X_{i},\beta_{ji})\]
\(\beta_{ji}\) is the canonical isomorphism. In Theorem \ref{1.11}, we view the fibered category \(F\) as a pseudo-functor. Here we persist this view point and natural isomorphisms \(\alpha_{f,g}\)'s are defined as in Theorem \ref{1.11}. Then for any pullback diagram 
\[\begin{tikzcd} 
X_{ij} \arrow[r,"pr_2"] \arrow[d,"pr_1"'] \arrow[dr, phantom, "\ulcorner", very near start] & X_j \arrow[d,"\iota_j"] \\ X_i \arrow[r,"\iota_i" below] & X
\end{tikzcd}\]
\(\beta_{ji}\) is defined to be 
\[pr_{1}^*\iota_{i}^*E=(E|X_i)|X_{ij}\xrightarrow[\cong]{\alpha_{pr_1,\iota_i}}(\iota_ipr_1)^*E=E|X_{ij}\xrightarrow[\cong]{\alpha_{pr_2,\iota_j}^{-1}}pr_{2}^*\iota_{j}^*E=(E|X_j)|X_{ij}\]

\begin{definition}
A fibered category \(F:\mathcal{D}\rightarrow \mathcal{C}\) over a site \(\mathcal{C}\) is a \textbf{prestack} (resp. \textbf{stack}) is for any covering \(\{\iota_i:X_i\rightarrow X\}\) the natural functor \(F^{-1}(X)\rightarrow F_{\mathrm{des}}(\{\iota_i:X_i\rightarrow X\})\) is \textbf{fully faithful} (resp. an \textbf{equivalence}).
\end{definition}

\begin{remark}\label{1.17}
\normalfont
Roughly speaking the axiom for prestacks is to say morphisms between descent datums form a sheaf. For groupoid fibrations, the morphisms are actually isomorphisms. A prestack is a stack iff every descent datum is effective.
\end{remark}

\section{The Geonetry of Triangles}
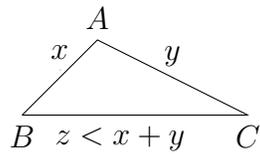
\begin{equation}
\notag
\begin{tikzpicture}[scale=1]
\draw[black] (-1,0) -- (2,0);
\draw[black] (-1,0) -- (0,1);
\draw[black] (0,1) -- (2,0);
\filldraw[black] (-1,0) circle (0pt) node[anchor=north]{\(B\)};
\filldraw[black] (0,1) circle (0pt) node[anchor=south]{\(A\)};
\filldraw[black] (2,0) circle (0pt) node[anchor=north]{\(C\)};
\filldraw[black] (-0.5,0.6) circle (0pt) node[anchor=south]{\(x\)};
\filldraw[black] (0.3,0) circle (0pt) node[anchor=north]{\(z<x+y\)};
\filldraw[black] (1,0.5) circle (0pt) node[anchor=south]{\(y\)};
\end{tikzpicture}
\end{equation}

\ A triangle is totally determined by lengths of its three edges. For any oriented triangle which means all of its three points \(A,\ B,\ C\) are fixed, it satisfies the following equations
\[M:=\begin{cases}
x+y>z>0\\
x+z>y>0\\
y+z>x>0
\end{cases}\]
In the three dimensional space \(\mathbb{R}^3\) it's an open cone $M$.
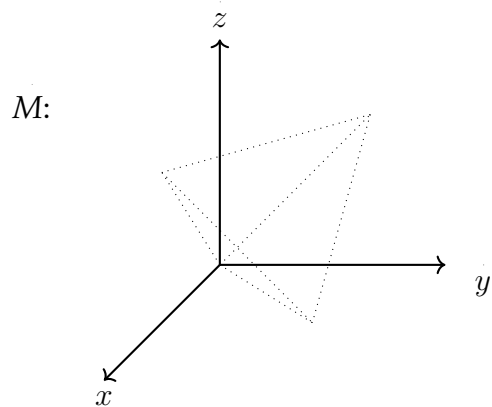
\begin{equation}\label{M}
\begin{tikzpicture}[scale=1]
\draw[thick,->] (0,0,0) -- (3,0,0);
\draw[thick,->] (0,0,0) -- (0,3,0);
\draw[thick,->] (0,0,0) -- (0,0,4);
\filldraw[black] (-2.5,2.4,0) circle (0pt) node[anchor=north]{\textit{M}:};
\filldraw[black] (3.5,0,0) circle (0pt) node[anchor=north]{\(y\)};
\filldraw[black] (0,3.5,0) circle (0pt) node[anchor=north]{\(z\)};
\filldraw[black] (0,0,4) circle (0pt) node[anchor=north]{\(x\)};
\draw[dotted] (2,2,0) -- (2,0,2);
\draw[dotted] (2,2,0) -- (0,2,2);
\draw[dotted] (0,2,2) -- (2,0,2);
\draw[dotted] (0,0,0) -- (2,0,2);
\draw[dotted] (0,0,0) -- (2,2,0);
\draw[dotted] (0,0,0) -- (0,2,2);
\end{tikzpicture}
\end{equation}

\ Two triangles correspond to the same point iff there is an oriented isometry between them. The boundary of this open cone represents degenerate triangles which are just line segments or the point. Moreover since this open cone is path connected, we see any two triangles can be continuously deformed to become each other. Note that intuitively every path in \(M\) defines an \textit{(oriented) continuous family of triangles}. Now for simplicity to draw a picture, we consider triangles up to similarity for a while. Then we have 
\[\begin{cases}
x+y>z>0\\
x+z>y>0\\
y+z>x>0\\
x+y+z=2
\end{cases}\]
which is just the equilateral triangle 
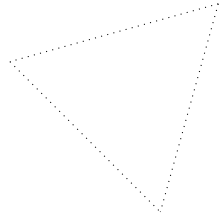
\begin{equation}
\notag
\begin{tikzpicture}[scale=1]
\draw[dotted] (2,2,0) -- (2,0,2);
\draw[dotted] (2,2,0) -- (0,2,2);
\draw[dotted] (0,2,2) -- (2,0,2);
\end{tikzpicture}
\end{equation}
and we also denote it by \(M\).

\ If we replace the relation of oriented isometries by arbitrary isometries, then \(M\) should be changed by 
\[N:=\{(x,y,z)\in \mathbb{R}^3|0<x\leq y\leq z,\ x+y>z\}\]
Its diagram up to similarity is 
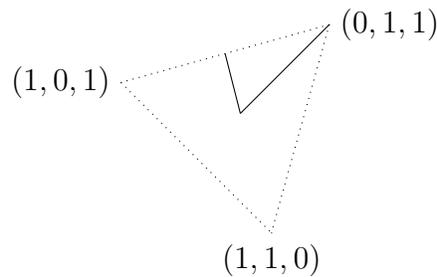
\begin{equation}
\notag
\begin{tikzpicture}[scale=1]
\draw[dotted] (2,2,0) -- (2,0,2);
\draw[dotted] (2,2,0) -- (0,2,2);
\draw[dotted] (0,2,2) -- (2,0,2);
\filldraw[black] (2,2,0) circle (0pt) node[anchor=west]{\((0,1,1)\)};
\filldraw[black] (2,0,2) circle (0pt) node[anchor=north]{\((1,1,0)\)};
\filldraw[black] (0,2,2) circle (0pt) node[anchor=east]{\((1,0,1)\)};
\draw[black] (1,2,1) -- (1.33,1.33,1.33);
\draw[black] (2,2,0) -- (1.33,1.33,1.33);
\end{tikzpicture}
\end{equation}
It's actually a right triangle and we denote it by \(N\) as well.
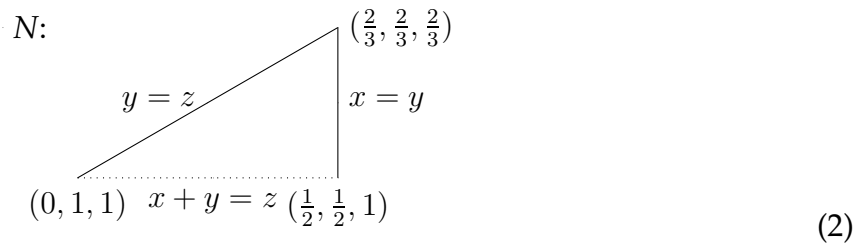
\begin{equation}\label{N}
\begin{tikzpicture}[scale=2]
\draw[dotted] (0,0) -- (1.73,0);
\draw[black] (1.73,1) -- (1.73,0);
\draw[black] (1.73,1) -- (0,0);
\filldraw[black] (0,0) circle (0pt) node[anchor=north]{\((0,1,1)\)};
\filldraw[black] (1.73,0) circle (0pt) node[anchor=north]{\((\frac{1}{2},\frac{1}{2},1)\)};
\filldraw[black] (1.73,1) circle (0pt) node[anchor=west]{\((\frac{2}{3},\frac{2}{3},\frac{2}{3})\)};
\filldraw[black] (1.73,0.5) circle (0pt) node[anchor=west]{\(x=y\)};
\filldraw[black] (0.9,0) circle (0pt) node[anchor=north]{\(x+y=z\)};
\filldraw[black] (0.86,0.5) circle (0pt) node[anchor=east]{\(y=z\)};
\filldraw[black] (-0.5,1) circle (0pt) node[anchor=west]{\textit{N}:};
\end{tikzpicture}
\end{equation}

\ Do spaces \(M\) and \(N\) contain all information of triangles especially families of triangles up to some type of equivalence? It's the question we mainly focus on in this section. Now we try to make the statement above precise.

\begin{definition}
A map \(f:X\rightarrow Y\) in \(\mathbf{Top}\) is \textbf{proper} if it's closed and \(\forall y\in Y,\ f^{-1}(y)\) is compact in \(X\).
\end{definition}

\ A map \(f:X\rightarrow Y\) is proper iff it's \textit{universally closed} which means for any pullback diagram 
\[\begin{tikzcd}
	{Z\times_Y X} & X \\
	Z & Y
	\arrow["f",from=1-2, to=2-2]
	\arrow["g"',from=2-1, to=2-2]
	\arrow[from=1-1, to=1-2]
	\arrow[from=1-1, to=2-1]
	\arrow["\ulcorner"{anchor=center, pos=0.125}, draw=none, from=1-1, to=2-2]
\end{tikzcd}\]
the map \(Z\times_YX\rightarrow Z\) will be closed. The proof can be found in [\href{https://stacks.math.columbia.edu/tag/005R}{Tag005R}]. Then it's easy to see proper maps are stable under pullbacks.

\begin{definition}
A \textbf{(non-oriented) continuous family of triangles} over a base space \(X\) is a continuous proper map \(\mathcal{F}\rightarrow X\) making \(\mathcal{F}\) be a fiber bundle over \(X\) such that there is a distance map \(d:\mathcal{F}\times_X \mathcal{F}\rightarrow \mathbb{R}_{\geq 0}\) whose restriction on \(\mathcal{F}_x\times \mathcal{F}_x\) is a metric on \(\mathcal{F}_x\) making it isometric to a triangle for any \(x\in X\). Moreover for every fiber \(\mathcal{F}_x\) we can associate it with a triple \((A_x,B_x,C_x)\) as endpoints of this triangle. There exists a global ordering $(A,B,C)$ such that the induced distance map
\[d'_{A,B} \ (d'_{A,C},\ \text{resp.}\ d'_{B,C}):X\rightarrow \mathbb{R}_{\geq 0},\ x\mapsto d(A_x,B_x)\ \  \big(d(A_x,C_x), \ \text{resp. } d(B_x,C_x)\big)\]
is a continuous map where \(d(A_x,B_x)\ \big(d(A_x,C_x), \ \text{resp. } d(B_x,C_x)\big)\)  comes from the distance map \(d:\mathcal{F}\times_X\mathcal{F}\rightarrow \mathbb{R}_{\geq 0}\). We say the pair $(\mathcal{F},A,B,C)$ is an \textbf{oriented continuous family of triangles}.
\end{definition}

\ We can define a morphism between two continuous families to be a continuous map over the base space. Then we associate every space \(X\) with a category consisting of continuous families of triangles over \(X\). But be careful here. In this category isomorphisms are just homeomorphisms, however, any two triangles are homeomorphic! In the ordinary case, to identify two triangles, we choose the notion of isometries instead of homeomorphisms. This case implies the category we defined above is not true here. Actually instead of obtaining a usual category, we try to obtain a groupoid according to our philosophy of identification, which proves to be more convenient.

\ So that an isomorphism \(\mathcal{F}\rightarrow \mathcal{G}\) between two families should induce isometries on fibers and then we will obtain a groupoid for every \(X\). It's the non-oriented viewpoint. As for the oriented viewpoint, isomorphisms of oriented families should moreover induce identity maps on those triples of endpoints. It's our main idea here and it provides a new way to think about the relation between \(M\) and \(N\). 

\ Any permutation of \(\{A,B,C\}\) induces an isometry on triangles. Since the permutation group of three elements is \(S_3\), viewing a triangle as a family over the one point space we may consider an action of \(S_3\) on \(M\) and it's not difficult to see the quotient space \(M/S_3\) is just \(N\). 

\ In the following let us consider the discrete case first, which means we do not view families of triangles over \(X\) as a category but just a set of isomorphism classes.

\begin{definition}
A \textbf{moduli problem} \(\mathcal{M}\) for topological spaces is a functor \(\mathbf{Top}^{op}\rightarrow \mathbf{Sets}\). A topological space \(M\) is called a \textbf{fine moduli spce} for \(\mathcal{M}\) if there is a natural isomorphism \(\mathcal{M}\xrightarrow{\sim}\mathrm{Hom}_{\mathbf{Top}}(-,M)\).
\end{definition}

\ We define a moduli functor for the problem of classifying (non-oriented) continuous families of triangles.
\[\mathcal{N}:\mathbf{Top}^{op}\rightarrow \mathbf{Sets},\ X\mapsto \{\text{continuous families of triangles over \(X\)}\}/\sim\]
For any continuous map \(f:Y\rightarrow X\), \(\mathcal{N}_f(\mathcal{F})\) is the pullback
\[\begin{tikzcd}
	{\mathcal{N}_f(\mathcal{F})} & {\mathcal{F}} \\
	Y & X
	\arrow["f"', from=2-1, to=2-2]
	\arrow[from=1-1, to=1-2]
	\arrow[from=1-2, to=2-2]
	\arrow[from=1-1, to=2-1]
	\arrow["\ulcorner"{anchor=center, pos=0.125}, draw=none, from=1-1, to=2-2]
\end{tikzcd}\]
Note that for every \(y\in Y\), the triangle \(\mathcal{N}_f(\mathcal{F})_y\) is actually \(\mathcal{F}_{f(y)}\). \(\mathcal{N}\) is then the functor for (non-oriented) families of triangles. As for oriented families, the moduli problem $\mathcal{M}$ is defined such that $\mathcal{M}(X)$ is the set of oriented continuous families of triangles up to isomorphism. \(\mathcal{M}_f\) is defined by pullbacks as well and \(\mathcal{M}_f(\mathcal{F})\) carries a natural orientation from \(\mathcal{F}\).

\begin{theorem}\label{2.4}
The moduli problem \(\mathcal{M}\) admits a fine moduli space, which is just the open cone \(M\) in the picture (\ref{M}).
\end{theorem}
\begin{proof}
There is a universal family \(\widetilde{M}\subseteq M\times \mathbb{R}^2\) of triangles over \(M\) whose fiber \(\widetilde{M}_m\) is just the triangle defined by endpoints \(A,\ B,\ C\) and edge lengths \(x,\ y,\ z\) described at the beginning of this section where \(m=(x,y,z)\in M\). And on the part of \(\mathbb{R}^2\) all points of \(A_m\) coincide. It's intuitive to see the projection \(\widetilde{M}\rightarrow M\) actually defines a continuous family of triangles which is left to readers. In the following we prove for any family \(\mathcal{F}\rightarrow X\) there is a unique map \(X\rightarrow M\) making \(\mathcal{F}\) be the pullback \(\mathcal{M}_f(\widetilde{M})\).

\ For \(p:\mathcal{F}\rightarrow X\), there is a unique map \(q:X\rightarrow M\) such that \(q(x)\in M\) corresponds exactly the oriented triangle \(\mathcal{F}_x\). We prove \(q\) is continuous first. Suppose \(Z\subseteq M\) is closed. \(q^{-1}(Z)=p(\underset{x\in q^{-1}(Z)}{\cup}\mathcal{F}_x)\). Since \(p\) is universally closed, it's enough to prove \(\underset{x\in q^{-1}(Z)}{\cup}\mathcal{F}_x\) is closed. This is not difficult to understand since \(\mathcal{F}\rightarrow X\) is a fiber bundle and in a neighborhood of \(\mathcal{F}\), triangles should not be so far away from each other. Formally speaking, we should choose a limit point \(u\in \mathcal{F}\) over some \(y\in X\) and to prove \(u\in \underset{x\in q^{-1}(Z)}{\cup}\mathcal{F}_x\), it's enough to prove the triangle \(\mathcal{F}_y\) is in the closed subset \(Z\).

\ Since distance maps 
\[d'_{A,B},\ d'_{A,C},\ d'_{B,C}:X\rightarrow \mathbb{R}_{\geq 0}\]
are all continuous, for any open subset \(V\subseteq M\) containg the triangle \(\mathcal{F}_y\) as a point, there will be an open neighborhood \(U\) of \(X\) containing \(y\) such that \(q(U)\subseteq V\). \(p^{-1}(U)\) is an open neighborhood of \(u\) hence containing some point of \(\underset{x\in q^{-1}(Z)}{\cup}\mathcal{F}_x\) which means some point of \(q^{-1}(Z)\) lies in \(U\). Then \(V\cap Z\ne \emptyset\) and since \(Z\) is a closed subset, as a point the triangle \(\mathcal{F}_y\) belongs to \(Z\).

\ The map \(\mathcal{F}\rightarrow \widetilde{M}\) can be defined on fibers. For every \(x\in X\), there is an isomoetry \(\mathcal{F}_x\rightarrow \widetilde{M}_{q(x)}\). All of these maps on fibers \(\mathcal{F}_x\) can be glued to be a global continuous map \(\mathcal{F}\rightarrow \widetilde{M}\).
\[\begin{tikzcd}
	{\mathcal{F}} & {\widetilde{M}\subseteq M\times \mathbb{R}^2} \\
	X & M
	\arrow["q"', from=2-1, to=2-2]
	\arrow["{pr_1}", from=1-2, to=2-2]
	\arrow["p"', from=1-1, to=2-1]
	\arrow[from=1-1, to=1-2]
	\arrow["\ulcorner"{anchor=center, pos=0.125}, draw=none, from=1-1, to=2-2]
\end{tikzcd}\]

\ To see \(\mathcal{F}\rightarrow \widetilde{M}\) is continuous we should realize \(\mathcal{F}\rightarrow X\) is a fiber bundle which means for any \(x\in X\) there will exist an open neighborhood \(U\subseteq X\) containing \(x\) such that 
\[\begin{tikzcd}
	{p^{-1}(U)} & {U\times T} \\
	& U
	\arrow["{pr_1}", from=1-2, to=2-2]
	\arrow["\approx", from=1-1, to=1-2]
	\arrow["p"', from=1-1, to=2-2]
\end{tikzcd}\]
The resulting map \(U\times T\rightarrow M\times \mathbb{R}^2\) is continuous. And then the canonical map \(\mathcal{F}\rightarrow X\times_M \widetilde{M}\) induces isometries on fibers hence being an isomorphism for oriented families of triangles.
\end{proof}

\begin{remark}\label{2.5}
\normalfont
Although \(M\) is a fine moduli space classifying oriented triangles, its quotient space \(M/S_3=N\) is not a fine moduli space classifying non-oriented triangles.
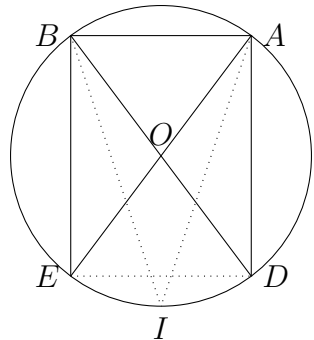
\[\begin{tikzpicture}[scale=2]
\draw (0,0) circle (1cm);
\filldraw[black] (0.6,0.8) circle (0pt) node[anchor=west]{\(A\)};
\draw[black] (-0.6,0.8) -- (0.6,0.8);
\filldraw[black] (-0.6,0.8) circle (0pt) node[anchor=east]{\(B\)};
\draw[dotted] (-0.6,-0.8) -- (0.6,-0.8);
\filldraw[black] (0,-1) circle (0pt) node[anchor=north]{\(I\)};
\filldraw[black] (0.6,-0.8) circle (0pt) node[anchor=west]{\(D\)};
\filldraw[black] (-0.6,-0.8) circle (0pt) node[anchor=east]{\(E\)};
\filldraw[black] (0,0) circle (0pt) node[anchor=south]{\(O\)};
\draw[black] (0.6,0.8) -- (0.6,-0.8);
\draw[black] (-0.6,0.8) -- (0.6,-0.8);
\draw[black] (0.6,0.8) -- (-0.6,-0.8);
\draw[black] (-0.6,0.8) -- (-0.6,-0.8);
\draw[dotted] (-0.6,0.8) -- (0,-1);
\draw[dotted] (0.6,0.8) -- (0,-1);
\end{tikzpicture}\]

\ Consider the picture above. \(I\) is the arc \(\wideparen{DE}\) and actually we may view \(I\) as the unit interval. For a triangle \(T\), its endpoints \(A,\ B\) are fixed and the third endpoint \(C\) lies in the arc \(I\). We define two families of triangles over \(I\).

\ The first family \(\mathcal{F}\subseteq I\times \mathbb{R}^2\) is defined such that the endpoint \(C\) goes from \(D\) to \(E\) continuously which means the beginning triangle is \(\triangle ABD\) and the end triangle is \(\triangle ABE\).

\ The second family \(\mathcal{G}\subseteq I\times \mathbb{R}^2\) is defined such that the end point \(C\) goes from \(D\) to the midpoint of the arc \(I\) first and then goes back from this midpoint to \(D\). This means the beginning and end triangles are all \(\triangle ABD\).

\ The two families are different since to construct an isomorphism \(\mathcal{F}\rightarrow \mathcal{G}\) when the endpoint \(C\) is in the first half part of the arc \(I\) , the induced map on fibers should be the identity map. But when \(C\) is in the second half part, the induced map on fibers should be the symmetry along the diameter passing through the origin \(O\) and the midpoint of \(I\). Such isometries on fibers can not be glued to a continuous global map \(\mathcal{F}\rightarrow \mathcal{G}\) since there is a dramtic change around the isosceles triangle whose end points consisting of \(A,\ B\) and the midpoint of \(I\).

\ Here the isomorphsm classes of \(\mathcal{F}\) and \(\mathcal{G}\) are different but they induce the same map \(I\rightarrow N\) since triangles \(\triangle ABD\) and \(\triangle ABE\) are identified in \(N\). Therefore \(N\) is not a fine moduli space. We will prove later that \(N\) is a so called \textit{coarse moduli space}.
\end{remark}

\begin{remark}\label{2.6}
\normalfont
From the Remark \ref{2.5} we know isosceles triangles prevent \(N\) from being a fine moduli space. Actually the moduli functor \(\mathcal{N}'\) classifying non-oriented \textit{scalene triangles} (all edges has different lengths) admits a fine moduli space \(N'\), whose picture up to similarity is 
\begin{equation}\label{N'}
\begin{tikzpicture}[scale=2]
\draw[dotted] (0,0) -- (1.73,0);
\draw[dotted] (1.73,1) -- (1.73,0);
\draw[dotted] (1.73,1) -- (0,0);
\filldraw[black] (0,0) circle (0pt) node[anchor=north]{\((0,1,1)\)};
\filldraw[black] (1.73,0) circle (0pt) node[anchor=north]{\((\frac{1}{2},\frac{1}{2},1)\)};
\filldraw[black] (1.73,1) circle (0pt) node[anchor=west]{\((\frac{2}{3},\frac{2}{3},\frac{2}{3})\)};
\filldraw[black] (1.73,0.5) circle (0pt) node[anchor=west]{\(x=y\)};
\filldraw[black] (0.9,0) circle (0pt) node[anchor=north]{\(x+y=z\)};
\filldraw[black] (0.86,0.5) circle (0pt) node[anchor=east]{\(y=z\)};
\filldraw[black] (-0.5,1) circle (0pt) node[anchor=west]{\textit{N}':};
\end{tikzpicture}
\end{equation}

\ A family \(\mathcal{F}\) of scalene triangles over \(X\) has a natural ordering \((A,B,C)\) such that \(A_xB_x<A_xC_x<B_xC_x\). An isomorphism between non-oriented families of scalene triangles preserves these natural orderings so that from Theorem \ref{2.4} we know the moduli functor classifying scalene triangles is fine.
\end{remark}

\ For a fine moduli functor \(\mathcal{M}\) its corresponding functor is actually \(P_M:\mathbf{Top}/M\rightarrow \mathbf{Top}\).

\begin{theorem}\label{2.7}
No matter for the usual topology or the etale topology described in Example \ref{1.14}, the category \(\mathbf{Top}/M\) over \(\mathbf{Top}\) is a stack.
\end{theorem}
\begin{proof}
\(P_M:\mathbf{Top}/M\rightarrow \mathbf{Top}\) is a groupoid fibration. Consider the diagram
\[\begin{tikzcd}
	{x'} & x & y & {:\mathbf{Top}/M} \\
	X & X & Y & {:\mathbf{Top}}
	\arrow["f"', from=2-2, to=2-3]
	\arrow["{f'}", from=1-2, to=1-3]
	\arrow["{P_M}", dotted, no head, from=1-3, to=2-3]
	\arrow[dotted, no head, from=1-2, to=2-2]
	\arrow[equal, from=2-1, to=2-2]
	\arrow["{f''}", bend left, from=1-1, to=1-3]
	\arrow[dotted, no head,from=1-1, to=2-1]
\end{tikzcd}\]
where \(x',x:X\rightarrow M\) and \(y:Y\rightarrow M\) represent morphisms with the target \(M\). Then \(f'\) is actually the map \(f\) and is unique. Therefor for any other lifting \(f''\) over \(f\), the unique lifting \(\mathrm{id}\) over \(\mathrm{id}_X\) satisfies \(f'\circ \mathrm{id}=f''=f\).

\ Whenever \(\mathbf{Top}\) is endowed with a subcanonical topology which means for any space \(X\), the representable functor \(\mathrm{Hom}_{\mathbf{Top}}(-,X)\) is a sheaf, \(\mathbf{Top}/M\) is a stack. Then this theorem follows from Lemma \ref{2.8}.

\ For any covering \(\{p_i:E_i\rightarrow X\}\), given a descent datum \((e_i,\alpha_{ji})\) where \(e_i:E_i\rightarrow M\) since \(\alpha_{ji}\) is over \(E_{ij}\) i.e. in the fiber of \(E_{ij}\), \(\alpha_{ji}=\mathrm{id}\). Then the cocycle condition is trivial. \(e_{i}|E_{ij}=e_j|E_{ij}\). From the sheaf axiom there will exist a unique morphism \(e:X\rightarrow M\) whose restriction on \(E_i\) is \(e_i\). This proves every descent datum is effective. The fully faithfulness follows from the sheaf axiom directly.
\end{proof}

\begin{lemma}\label{2.8}
With the etale topology for any space \(Y\), the representable functor \(\mathrm{Hom}_{\mathbf{Top}}(-,Y):\mathbf{Top}\rightarrow \mathbf{Sets}\) is a sheaf.
\end{lemma}
\begin{proof}
Given an etale covering \(\{p_i:E_i\rightarrow X\}\) and maps \(f_i:E_i\rightarrow Y\) such that \(f_i|E_{ij}=f_{j}|E_{ij}\), there will be a set theoretical map \(f:X\rightarrow Y\) such that for any \(x=p_i(e_i)\) in \(X\), \(f(x)=f_i(e_i)\). It's well defined since pullbacks of \(f_i\)'s coincide. This analysis also implies \(f\) is unique.

\ To prove \(f\) is continuous suppose \(V\subseteq Y\) is open. Then \(f^{-1}(V)=\cup_ip_i(f_{i}^{-1}(V))\) is open since an etale map (local homeomorphism) is an open map.
\end{proof}

\begin{definition}\label{2.9}
A moduli functor \(\mathcal{N}:\mathbf{Top}^{op}\rightarrow \mathbf{Sets}\) admits a \textbf{coarse moduli space} \(N\) if there is a natural transformation \(\alpha:\mathcal{N}\rightarrow \mathrm{Hom}_{\mathbf{Top}}(-,N)\) such that 
\item[(1)] \(\alpha(*):\mathcal{N}\xrightarrow{\sim} \mathrm{Hom}_{\mathbf{Top}}(*,N)=N\).
\item[(2)] For any another transformation \(\beta:\mathcal{N}\rightarrow \mathrm{Hom}_{\mathbf{Top}}(-,Y)\) there is a unique natural transformation \(\theta_*:\mathrm{Hom}_{\mathbf{Top}}(-,N)\rightarrow \mathrm{Hom}_{\mathbf{Top}}(-,Y)\) induced by \(\theta:N\rightarrow Y\) satisfying \(\theta_*\circ \alpha=\beta\).
\[\begin{tikzcd}
	{\mathcal{N}} & {\mathrm{Hom}_{\mathbf{Top}}(-,N)} \\
	{\mathrm{Hom}_{\mathbf{Top}}(-,Y)}
	\arrow["\beta"', from=1-1, to=2-1]
	\arrow["\alpha", from=1-1, to=1-2]
	\arrow["{\exists!\ \theta_{*}}", dotted, from=1-2, to=2-1]
\end{tikzcd}\]
\end{definition}

\ From the definition we see coarse moduli spaces are unique.

\begin{theorem}\label{2.10}
The moduli functor \(\mathcal{N}\) classifying non-oriented triangles admits the coarse moduli space \(N\) described in the picture (\ref{N}).
\end{theorem}
\begin{proof}
For any concrete (not up to isomorphism) continuous family \(\mathcal{F}\) of triangles over \(X\) since it admits an ordering, it induces a map \(X\rightarrow M\). Composed with \(M\rightarrow M/S_3=N\), the map \(X\rightarrow N\) is well defined for the isomorphism class of \(\mathcal{F}\). This defines a natural transformation \(\alpha:\mathcal{N}\rightarrow \mathrm{Hom}_{\mathbf{Top}}(-,N)\) and clearly \(\alpha(*)\) is a bijection. Via \(\alpha\) the subfamily \(\widetilde{N}=\widetilde{M}\cap (N\times \mathbb{R}^2)\) is sent to the identity map \(\mathrm{id}_N\). Such family over \(N\) is called a \textit{modular family}.

\ Given any other transformation \(\beta:\mathcal{N}\rightarrow \mathrm{Hom}_{\mathbf{Top}}(-,Y)\), if \(\theta_*\) exists then it must send \(\mathrm{id}_N\) to \(\beta(\widetilde{N})\in \mathrm{Hom}_{\mathbf{Top}}(N,Y)\) which means \(\theta=\beta(\widetilde{N})\) by Yoneda's lemma. Therefore we only need to prove \(\beta(\widetilde{N})_*\) making such a diagram in Definition \ref{2.9} commutative.

\ Over the one point space 
\[\begin{tikzcd}
	{\mathcal{N}(*)} & N \\
	Y
	\arrow["\alpha","\cong" below, from=1-1, to=1-2]
	\arrow["\beta"', from=1-1, to=2-1]
	\arrow["\mu", dotted, from=1-2, to=2-1]
\end{tikzcd}\]
since \(\alpha(*)\) is a bijection there exists a unique set theoretical map \(\mu:N\rightarrow Y\) making the diagram commutative. For any pointed space \((X,x)\), it represents a function \(x:*\rightarrow X\).
\[\begin{tikzcd}
	{\mathcal{F}\in\mathcal{N}(X)} & {\mathrm{Hom}_{\mathbf{Top}}(X,N)} \\
	{\mathrm{Hom}_{\mathbf{Top}}(X,Y)} & {\mathcal{N}(*)} & N \\
	& Y
	\arrow["\beta"', from=2-2, to=3-2]
	\arrow[from=1-1, to=2-2]
	\arrow["\beta"', from=1-1, to=2-1]
	\arrow["{x^*}"', from=2-1, to=3-2]
	\arrow["\alpha", from=2-2, to=2-3]
	\arrow["\mu", dotted, from=2-3, to=3-2]
	\arrow["\alpha", from=1-1, to=1-2]
	\arrow["{x^*}", from=1-2, to=2-3]
	\arrow["{\mu_*}"{description, pos=0.7}, dotted, from=1-2, to=2-1]
\end{tikzcd}\]

\ We want to prove \(\mu\circ \alpha(\mathcal{F})=\beta(\mathcal{F})\). But since 
\begin{align*}
    x^*\circ \mu_*\circ \alpha&=\mu\circ x^{*}\circ \alpha\\
    &=\mu\circ \alpha\circ \mathcal{N}(x)\\
    &=\beta\circ \mathcal{N}(x)\\
    &=x^*\circ \beta
\end{align*}
this means \(\mu\circ \alpha(\mathcal{F})\) and \(\beta(\mathcal{F})\) have the same value on \(x\in X\),
\[\begin{tikzcd}
	{*} & X & N & Y
	\arrow["x", from=1-1, to=1-2]
	\arrow["{\alpha(\mathcal{F})}", from=1-2, to=1-3]
	\arrow["\mu", from=1-3, to=1-4]
	\arrow[bend right,"\beta(\mathcal{F})" below, from=1-2, to=1-4]
\end{tikzcd}\]
and we obtain \(\mu\circ \alpha(\mathcal{F})=\beta(\mathcal{F})\). Moreover let \(\mathcal{F}=\widetilde{N}\) and then \(\alpha(\widetilde{N})=\mathrm{id}_N\Rightarrow \mu=\beta(\widetilde{N})\) is continuous.
\end{proof}

\begin{remark}\label{2.11}
\normalfont
From Remark \ref{2.6} and Theorem \ref{2.10} we may guess the moduli functor classifying isosceles triangles admits the coarse moduli space. Its coarse moduli space is actually an open inverval such as \((0,\pi)\). But note that in general modular families are not unique.\footnote{See \cite[Sec. 1.5]{Beh}.}
\end{remark}

\subsection{Quotient Stacks}
In algebraic topology for any group \(G\) there exists its \textit{classifying space} \(K(G,1)=BG\) unique up to homotopy such that there are natural bijections 
\[[X,BG]\xrightarrow{\sim}H^1(X;G)\]
for any CW-complex \(X\) where \([X,BG]\) denote the homotopy classes of continuous maps from \(X\) to \(BG\). The base space \(BG\) admits the \textit{universal covering} \(p:EG\rightarrow BG\). Moreover we have isomorphisms of groups
\[\pi_1(BG,*)\cong \mathrm{Aut}(p)\cong G\]
so that the space \(EG\) carries a free action (on the right) of the group \(G\) by covering transformation. Via this action \(G\) acts transitively on fibers \(p^{-1}(*)\). Then \(p\) is the quotient map by \(G\) and \(BG=EG/G\).

\ In tradition if \(G\) acts on \(X\) (on the left) non-freely, we does not define the quotient space of \(X\) to be \(X/G\) in the usual sense. In fact we define the quotient space \([X/G]\) to be the quotient of \(EG\times X\) by the diagonal action \(g\cdot (y,x)=(y\cdot g^{-1},g\cdot x)\) so that the quotient map \(EG\times X\rightarrow [X/G]\) is a \textit{principal \(G\)-bundle}.
\begin{definition}
Suppose a topological group \(G\) acts trivially on \(X\). A \textbf{principal \(G\)-bundle} or \textbf{\(G\)-torsor} is a continuous map \(p:E\rightarrow X\) where \(E\) is a non-empty \(G\)-space with the multiplication \(\rho:E\rightarrow G\rightarrow X\) such that 
\item[(1)] the diagram 
\[\begin{tikzcd}
E\times G\arrow[r,shift left,"\rho" above]\arrow[r,shift right,"pr_1" below]&E\arrow[r,"p"]&X
\end{tikzcd}\]
is a coequalizer.
\item[(2)] \(p\) is locally trivial i.e. there is an open covering \(\{U_i\rightarrow X\}\) such that 
\[\begin{tikzcd}
	{p^{-1}(U_i)} & {U_i\times G} \\
	& {U_i}
	\arrow["\approx", from=1-1, to=1-2]
	\arrow["{pr_1}", from=1-2, to=2-2]
	\arrow["p"', from=1-1, to=2-2]
\end{tikzcd}\]
\end{definition}

\ A classical theorem of algebraic topology asserts that \(BG\) classifies principal \(G\)-bundle over CW-complexes which means for any principal \(G\)-bundle \(E\) over a CW-complex \(X\) there will exist a map \(X\rightarrow BG\) unique up to homotopy such that \(E\approx X\times_{BG}EG\). That's the main reason why we define the quotient space for a non-free action of \(G\) on \(X\) by \([X/G]\). If \(X=*\) is the one point space, \(*/G=*\) but \([*/G]=BG\) contains the information of principal bundles.

\begin{lemma}\label{2.13}
Any morphism of principal \(G\)-bundles is an isomorphism.
\end{lemma}
\begin{proof}
Given a map \(f:E\rightarrow E'\) of principal \(G\)-bundles, we suppose \(E=E'=X\times G\) is trivial first.
\[\begin{tikzcd}
	{X\times G} & {X\times G} \\
	& X
	\arrow["\approx" below,"f" above, from=1-1, to=1-2]
	\arrow["{pr_1}", from=1-2, to=2-2]
	\arrow["{pr_1}"', from=1-1, to=2-2]
\end{tikzcd}\]
Then \(f(x,g)=(x,u(x)\cdot g)\). Clearly \(u:X\rightarrow G\) is a continuous map. Then the inverse of \(f\) is defined to be \((x,g)\mapsto (x,u(x)^{-1}\cdot g)\). Hence the equivariant map \(f\) is an isomorphism. Since principal \(G\)-bundles are locally trivial, in general \(f\) is an isomorphism.
\end{proof}

\ For any space \(X\) we associate it with the groupoid of principal \(G\)-bundles over \(X\) denoteed by \(\mathcal{P}_GX\). If \(G\) is clear, then it's written as \(\mathcal{P}X\) simply. For any map \(f:Y\rightarrow X\) the functor \(\mathcal{P}(f):\mathcal{P}X\rightarrow \mathcal{P}Y\) is defined by the pullback along \(f\). Then it's especially a pseudo-functor. From Theorem \ref{1.11}, its fibered category \(BG\rightarrow \mathbf{Top}\), where we use the symbol \(BG\) to denote the category \(G\)-torsors in tradition, consists of objects principal \(G\)-bundles \(E\rightarrow X\) and its arrows are commutative diagrams 
\[\begin{tikzcd}
	{E'} & E \\
	{X'} & X
	\arrow[from=1-1, to=2-1]
	\arrow[from=2-1, to=2-2]
	\arrow[from=1-1, to=1-2]
	\arrow[from=1-2, to=2-2]
\end{tikzcd}\]
with an isomorphism \(E'\xrightarrow{\sim} X'\times_X E\). Note that \(E'\rightarrow E\) is a \(G\)-equivariant map.

\begin{theorem}\label{2.14}
The fibered category \(BG\rightarrow \mathbf{Top}\) of \(G\)-torsors is a stack.
\end{theorem}
\begin{proof}
Given an open covering \(\{\iota_i:U_i\rightarrow X\}\) and a morphism \(f:(E_i,\alpha_{ji})\rightarrow (E'_i,\beta_{ji})\) of descent datums such that \((E_i,\alpha_{ji})\) and \((E'_i,\beta_{ji})\) are induced by two global principal \(G\)-bundles \(E\), \(E'\) over \(X\) respectively, then \(\alpha_{ji}\) and \(\beta_{ji}\) should be idnrity maps. Since any representable functor over \(\mathbf{Top}\) with the usual topology is a sheaf, from the sheaf axiom \(f_i\)'s can be glued to a unique map \(f:E\rightarrow E'\).

\ Given a descent datum \((E_i,\alpha_{ji})\) over an open covering \(\{\iota_i:U_i\rightarrow X\}\), we define \(E=\coprod E_i/\sim\) where \(e_i\in E_i\sim e_j\in E_j\) if \(\alpha_{ji}(e_i)=e_j\). The induced map \(p:E\rightarrow X\) is a principal \(G\)-bundle since axioms of \(G\)-torsors are local and its restriction on \(U_i\) is actually \(E_i\).
\end{proof}

\ In tradition for a principal \(G\)-bundle \(E\) over \(Y\), \(G\) acts on the right. Now we suppose for a given space \(X\), \(G\) acts on the left. Then a \(G\)-equivariant map \(f:E\rightarrow X\) should satisfy \(f(e\cdot g)=g^{-1}\cdot f(e)\). Actually for the left \(G\)-space \(X\) it has a natural right \(G\)-space structure defined by \(x\cdot g=g^{-1}\cdot x\).

\ For a left \(G\)-space \(X\) the quotient stack \([X/G]\) is defined as follows its objects are principal \(G\)-bundles \(E\rightarrow Y\) with a \(G\)-equivariant map \(E\rightarrow X\) and morphisms are commutative diagrams 
\[\begin{tikzcd}
	{E'} & E \\
	{Y'} & Y
	\arrow[from=1-1, to=2-1]
	\arrow[from=2-1, to=2-2]
	\arrow[from=1-1, to=1-2]
	\arrow[from=1-2, to=2-2]
\end{tikzcd}\]
with an isomorphism \(E'\xrightarrow{\sim} Y'\times_Y E\) such that the composition of \(E'\rightarrow E\rightarrow X\) is just the \(E'\rightarrow X\). Then it's not difficult to see \([X/G]\) is actually a stack and moreover \([*/G]=BG\).

\ As explained before the moduli functor \(\mathcal{N}\) classifying non-oriented triangles is actually a groupoid fibration whose objects are continuous families of non-oriented triangles and morphisms are diagrams 
\[\begin{tikzcd}
	{\mathcal{F}'} & {\mathcal{F}} \\
	{X'} & X
	\arrow[from=1-1, to=2-1]
	\arrow[from=2-1, to=2-2]
	\arrow[from=1-1, to=1-2]
	\arrow[from=1-2, to=2-2]
\end{tikzcd}\]
where \(\mathcal{F}'\rightarrow \mathcal{F}\) induces isometries on fibers so that the natural map \(\mathcal{F}'\rightarrow X'\times_X\mathcal{F}\) is an isomorphism between families of non-oriented triangles.

\begin{theorem}\label{2.15}
The fibered category \(\mathcal{N}\) is equivalent to the quotient stack \([M/S_3]\).
\end{theorem}
\begin{proof}[Sketch of the proof]
For a space $X$ with an action of the group $G$, its quotient stack $[X/G]$ is actually the coequalizer for 
$$\begin{tikzcd}
	{X\times G} & X & {}
	\arrow["\rho", shift left, from=1-1, to=1-2]
	\arrow["{pr_1}"', shift right, from=1-1, to=1-2]
\end{tikzcd}$$
in the category of stacks. And we just need to notice that $\mathcal{N}$ is the quotient of $\mathcal{M}$ with the action $S_3$ when viewed as stacks.
\end{proof}

\subsection{Deformation Theory for Triangles}
\begin{definition}
Let \((X,x_0)\) be a pointed space and then a \textbf{deformation} of a triangle \(T\) over \((X,x_0)\) is defined to be a pair of maps 
\[T\xrightarrow{i}\mathcal{F}\xrightarrow{f}X\]
such that 
\item[(1)] \(fi(T)=x_0\).
\item[(2)] There exists an open neighborhood \(U\subseteq X\) containing \(x_0\) such that \(f^{-1}(U)\rightarrow U\) is a continuous family of triangles.
\item[(3)] The map \(i:T\rightarrow f^{-1}(x_0)\) is an isometry.

\ Two deformations 
\[T\xrightarrow{i}\mathcal{F}\xrightarrow{f}X,\ \ \ T\xrightarrow{j}\mathcal{G}\xrightarrow{g}X\]
of the triangle \(T\) over \((X,x_0)\) are equivalent if there exist an open neighborhood \(U\subseteq X\) containing \(x_0\) and the following commutative diagram 
\[\begin{tikzcd}
	T & {f^{-1}(U)} \\
	{g^{-1}(U)} & U
	\arrow["g"', from=2-1, to=2-2]
	\arrow["f", from=1-2, to=2-2]
	\arrow["j"', from=1-1, to=2-1]
	\arrow["i", from=1-1, to=1-2]
	\arrow[dotted, from=1-2, to=2-1]
\end{tikzcd}\]
where the diagonal map is an isomorphism between families of triangles.
\end{definition}

\ For every pointed space \((X,x_0)\) we denote the set of equivalent classes of deformations of \(T\) by \(\mathrm{Def}_T(X,x_0)\). Then it's clear for any open neighborhood \(U\subseteq X\) containing \(x_0\), \(\mathrm{Def}_T(X,x_0)=\mathrm{Def}_T(U,x_0)\).

Given a deformation \(\xi:T\xrightarrow{i}\mathcal{F}\xrightarrow{f}X\) for any pointed map \(g:(Y,y_0)\rightarrow (X,x_0)\) the pullback \(g^*\xi\) is defined to be 
\[g^*\xi:T\xrightarrow{(i,y_0)}\mathcal{F}\times_X Y\xrightarrow{pr_2}Y\]
This induces a well-defined morphism \(\mathrm{Def}_T(X,x_0)\rightarrow \mathrm{Def}_T(Y,y_0)\).

\end{document}